\documentclass[11pt,a4paper]{article}

\usepackage[hmargin={25mm,25mm},vmargin={25mm,25mm}]{geometry}

\usepackage{verbatim, ifpdf}
\usepackage{graphicx, color}
\usepackage{multirow}
\usepackage{amsthm}
\usepackage{amssymb, amsmath}
\usepackage{caption}
\usepackage{subcaption}
\usepackage{afterpage}
\usepackage{enumerate}
\usepackage{enumitem}
\usepackage{empheq}
\usepackage{scalerel}
\usepackage{multicol}
\usepackage{yhmath} 
\usepackage{stmaryrd}
\usepackage{anyfontsize}
\usepackage{wrapfig}
\usepackage{verbatim}
\usepackage{url,doi}

\usepackage{bbm}
\usepackage{tikz}
\usetikzlibrary{calc}
\usetikzlibrary {decorations}
\usetikzlibrary{positioning}

\usepackage{pgfplots}
\pgfplotsset{compat=1.14}
\usepackage{underscore}
\usepackage[capitalize]{cleveref}

\newcommand{\vc}[1]{\boldsymbol{#1}}

\newcommand{\ZZ}{\mathbb{Z}}
\newcommand{\NN}{\mathbb{N}}

\DeclareMathOperator{\lcm}{lcm}

\newcommand{\oto}[1]{\overset{#1}{\to}}

\newcommand{\abs}[1]{\left|#1\right|}

\newcommand{\res}{\!\upharpoonright}

\theoremstyle{definition}
\newtheorem*{theorem*}{Theorem}
\newtheorem{theorem}{Theorem}[section]
\newtheorem{lemma}[theorem]{Lemma}
\newtheorem{corollary}[theorem]{Corollary}

\newtheorem{question}[theorem]{Question}
\newtheorem{definition}[theorem]{Definition}
\newtheorem{proposition}[theorem]{Proposition}
\newtheorem{example}[theorem]{Example}

\newtheorem{remark}[theorem]{Remark}

\theoremstyle{plain}
\newtheorem*{claim*}{Claim}

\begin{document}
 \title{Graphs with Independent Exact $r$-covers for all $r$}
 \author{Hou Tin Chau\footnote{School of Mathematics, University of Bristol. Research supported by an Engineering and Physical Sciences Research Council (EPSRC) studentship.}}
\date{April 2025}
\maketitle

\begin{abstract}
For every natural number $d$, we construct finite $d$-regular simple graphs that, for every $r \le d$, contain an independent exact $r$-cover. This answers a question of Gray and Johnson that arose in their study of 2-step transit probabilities.

We obtain some divisibility conditions on the order $n$ of graphs that for every $r \le d$ contain an independent exact $r$-cover, and give constructions for $d=3, 4, 5, 6$ where the order of the graph is minimal (we deduce this minimality from our divisibility conditions). We construct these graphs as common coverings of smaller graphs. We revisit a result of Angluin and Gardiner on finite common coverings of two regular graphs of the same degree, and the result of Gross that regular graphs of even degree are Schreier coset graphs. We combine both results to provide a finite common covering of two regular graphs of the same degree, that uses fewer vertices than the construction of Angluin and Gardiner in some cases. 
\end{abstract} 

\section{Introduction}

In \cite{thequestion}, Gray and Johnson looked at 2-step transit probabilities of random walks on 2-coloured regular graphs. 

\begin{definition}[2-step transit probability, \cite{thequestion} Section 1]
Let $G$ be a finite $d$-regular graph with an even number of vertices, and $V(G) = R \sqcup B$ be a 2-colouring of the vertices that is \emph{balanced} ($\abs{R} = \abs{B}$). Consider the random walk on $G$ where each step we move to a uniformly random neighbour of the current vertex. Define $P_2(R)$ as the probability that, starting from a uniformly random vertex $v_0$ in $R$, the next two steps $v_1, v_2$ both stay in $R$. Similarly define $P_2(B)$. 
\end{definition}

Gray and Johnson raise the question of describing the subset of $[0, 1]^2$ consisting of all possible values the pair $(P_2(R), P_2(B))$ can take. Defining the region $D_{d}$ to be the convex hull of 

\[\{(0, 0)\} \cup \left\{ \left( \frac{l}{d}, \frac{l^2}{d^2}\right): l \in [d-1]  \right\} \cup
\left\{ \left( \frac{l^2}{d^2}, \frac{l}{d}\right): l \in [d-1]  \right\}  \cup \{(1, 1)\},\]

they showed that, for all $d$-regular graphs $G$ and all balanced 2-colourings, $(P_2(R), P_2(B))$ lies in $D_{d}$  (\cite{thequestion} Theorem 3), and in fact if there exists a finite graph attaining all extreme points of $D_{d}$, then $D_{d}$ is the closure of the set of possible pairs $(P_2(R), P_2(B))$ (\cite{thequestion} Question 12 and Section 5).

They found that the extreme points of $D_{d}$ are closely related to independent exact $r$-covers.

\begin{definition}[Independent exact $r$-cover \cite{thequestion}, p.7]
\label{def:indepexact}
Let $G$ be a $d$-regular graph. For $0 \le r \le d$, an \emph{independent exact $r$-cover} is a subset $S$ of vertices of $G$ such that there are no edges within $S$ and every vertex in $V(G) \setminus S$ is adjacent to exactly $r$ vertices in $S$. (In other words, $S$ and $V\setminus S$ form an \emph{equitable partition} with matrix $\begin{pmatrix}0 & d \\ r & d-r\end{pmatrix}$ in the sense of \cite{equitable} p.159.)

When $\abs{G}$ is finite, by double counting the number of edges between $S$ and $V\setminus S$ it is easy to see that any independent exact $r$-cover has size $r\abs{G}/(d+r)$.

(Note that the empty set is always an independent exact $0$-cover, so we only consider $r \ge 1$ in what follows.)

\end{definition}

The connection can be seen from the following example: if $G$ has an independent exact $r$-cover, then $(2d)\cdot G$, the graph consisting of $2d$ disjoint copies of $G$, can attain the extreme point 
\[\left(\frac{d-r}{d}, \quad \frac{(d-r)^2}{d^2}\right)
\]
of $D_d$, by taking the red subset to be $(d-r)$ copies of $G$ and $(d+r)$ copies of $S$, and the blue subset to be the complement, which is $(d+r)$ copies of $G\setminus S$.

Gray and Johnson asked the following.

\begin{question}[\cite{thequestion}, Question 13]
\label{ques:indepexact}
For each natural number $d$, is there a finite $d$-regular graph $G$ such that, for every $r \in [d]$, there is an independent exact $r$-cover $S_r \subset V(G)$?
\end{question}

They described a construction for $d = 3$, on 40 vertices. This paper answers the question in the affirmative for general $d$.

\begin{theorem}
\label{thm:main}
For each natural number $d$, there exists a finite $d$-regular simple graph $G$ with an independent exact $r$-cover $S_r \subset V(G)$ for all $r \in [d]$. 
\end{theorem}

In \cref{combine} we observe that it suffices to find graphs satisfying the condition for each $r$ individually, because it is possible to combine them by taking a common covering\footnote{
We distinguish between a ``covering'' of a graph (in the topological sense, \cref{def:covering}) and a ``cover'' of a graph (every other vertex is adjacent to a vertex from the cover). In particular we only consider the special case of ``independent exact $r$-cover'' (\cref{def:indepexact}), which we always refer to by the full name. }. In \cref{smallgraphs} we provide these smaller graphs. In \cref{divisibility}, we study how independent exact $r$-covers for different $r$ intersect, leading to divisibility conditions on the graph order $n$ which proves optimality of constructions we have for $d=3, 4, 5, 6$.

\section{Graph coverings}
\label{combine}
In constructing graphs satisfying \cref{thm:main}, it is convenient to allow multiple edges, self-loops, and semi-edges in the intermediate steps. We shall call them ``generalized graphs'' to distinguish from the simple graphs without multiple edges, loops or semi-edges. As we shall soon observe, the common covering construction for \cref{thm:main} that we obtain from these generalized graphs would still be simple. 

\begin{definition}[\cite{semiedge} Definition 1]
A \emph{generalized graph} consists of some vertices (forming the set $V$) and some \emph{darts} (one half of an edge). Each dart is incident to exactly one vertex. Some darts are paired to another dart to form an \emph{edge}. The \emph{dart neighbourhood} $\Gamma(v)$ of a vertex $v$ is the set of darts incident to it, and the \emph{degree} of $v$ is the cardinality of this set. Each dart falls into one of the following three cases depending on whether/how it is paired to other darts: 

\begin{enumerate}
    \item If two darts incident to distinct vertices $u, v$ are paired, they form an \emph{ordinary edge} between $u$ and $v$. 
    \item If two distinct darts incident to the same vertex $v$ are paired, they form a \emph{loop} (this counts as degree $2$ to $v$).
    \item If a dart is unpaired (equivalently, paired to itself), it is a \emph{semi-edge} (counts as degree $1$).
\end{enumerate}

We allow multiple edges between the same pair of vertices, and also multiple loops and semi-edges on the same vertex.

\end{definition}

Semi-edges arise naturally when the two endpoints of an ordinary edge are mapped to the same vertex (\cite{semiedge} Section 1.2).

\begin{definition}[Generalized graph covering, \cite{semiedge} Definition 4]
\label{def:covering}
Let $G$, $H$ be generalized graphs. A map between the sets of darts $\pi\colon D(G) \to D(H)$ is said to be a {\em covering map} if all of the following holds:
\begin{enumerate}
    \item $\pi$ is surjective onto $D(H)$.
    \item \label{def:vertexmap} If two darts $d_1, d_2$ are incident to the same vertex $v \in V(G)$, then $\pi(d_1), \pi(d_2)$ are incident to the same vertex in $V(H)$ (which we shall denote as $\pi(v)$, thus defining $\pi \colon V(G) \to V(H)$).
    \item If two darts $d_1, d_2$ are paired in $G$, then $\pi(d_1), \pi(d_2)$ are either paired or are the same dart. If a dart $d$ is unpaired, then $\pi(d)$ is also unpaired.
    \item \label{def:localbij} $\pi$ restricts to a bijection
    \[\pi \res_{\Gamma(v)}: \Gamma(v) \to \Gamma(\pi(v))\] 
    on the dart neighbourhood of each vertex $v \in V(G)$.
\end{enumerate}
\end{definition}

\begin{remark}
    For simple graphs, the map on the vertex sets $\pi\colon V(G) \to V(H)$ (from \cref{def:vertexmap} above) is enough to specify where each edge goes under the covering map, and \cref{def:localbij} says $\pi$ restricts to a bijection on each \emph{vertex} neighbourhood. Thus, we recover the more familiar definition of a covering map (as can be found in \cite{GTM207} p.115), namely that it is a surjection from $V(G)$ to $V(H)$ that restricts to a bijection on the vertex neighbourhood of every $v\in V(G)$.
\end{remark}

Abusing notation slightly, we will sometimes write a covering map as $\pi\colon G \to H$, though formally it is defined as a map from $D(G)$ to $D(H)$.

Some standard properties of graph coverings are:
\begin{proposition}[Image of an edge, \cite{semiedge} Proposition 5 Items 5--7]
    \label{prop:img}
    Under a graph covering map $\pi$, 
    \begin{enumerate}
        \item the image of an ordinary edge can be an ordinary edge, a loop, or a semi-edge, but
        \item the image of a loop must be a loop, and 
        \item the image of a semi-edge must be a semi-edge.
    \end{enumerate}
\end{proposition}

\begin{proposition}[Preimage of an edge, \cite{semiedge} Proposition 5 Items 5--7]
    \label{prop:preimg}
    Under a graph covering map $\pi$, 
    \begin{enumerate}
        \item the preimage of an ordinary edge between two vertices $u, v$ is a matching between $\pi^{-1}(u)$ and $\pi^{-1}(v)$, 
        \item the preimage of a loop at $u$ is a union of vertex-disjoint cycles spanning $\pi^{-1}(u)$, where we allow 1-cycles (loops) and 2-cycles (two parallel edges), and 
        \item the preimage of a semi-edge at $u$ is a 1-factor of $\pi^{-1}(u)$, i.e.\ a union of some vertex-disjoint ordinary edges (a matching) and some semi-edges, which together span $\pi^{-1}(u)$.
    \end{enumerate}
\end{proposition}

These combinatorial coverings for generalized graphs differ from topological coverings only on how semi-edges are treated. When there are no semi-edges, the two concepts coincide if we view graphs as one-dimensional simplicial complexes (\cite{semiedge} Section 1.2). Semi-edges are however considered non-contractible: walking along a semi-edge is a non-trivial closed walk, but doing twice is a trivial closed walk, and the graph fundamental group (whose elements are reduced closed walks based at a vertex) of a single semi-edge is $\ZZ/2\ZZ$ (see \cite{coveringbook} pp. 4--5), consistent with the fact that an ordinary edge (with trivial fundamental group) gives a 2-fold covering over a semi-edge.

These notions are useful to \cref{ques:indepexact} for the following connection between covering and independent exact $r$-covers.

\begin{definition}[Independent exact $r$-cover for generalized graph]
    A vertex subset $S \subseteq V$ of a generalized graph is an \emph{independent set} if there are no loop and no semi-edge incident to any vertex in $S$, and no ordinary edges between any two vertices in $S$.

    It is an \emph{exact $r$-cover} if for every $v \in V \setminus S$, there are exactly $r$ ordinary edges from $v$ to $S$ (some of which might be incident to the same vertex in $S$).
\end{definition}

\begin{lemma} 
\label{lem:project}
Let $G$ and $H$ be $d$-regular generalized graphs and suppose $\pi\colon G \to H$ is a covering map. Suppose $S \subset V(H)$ is an independent exact $r$-cover of $H$. Then it can be lifted to an independent exact $r$-cover $\pi^{-1}(S)$ of $G$ .
\end{lemma}
\begin{proof}
    If $x \in \pi^{-1}(S)$ has any loop or semi-edge, then by \cref{prop:img}, $\pi(x) \in S$ also has a loop or semi-edge, contradicting that $S$ is independent. So $\pi^{-1}(S)$ has no loop or semi-edge. Similarly for distinct $x, y \in \pi^{-1}(S)$,  we have $\pi(x), \pi(y) \in S$ (though not necessarily distinct), and the image of any edge between $x, y$ would be either an ordinary edge between $\pi(x)$ and $\pi(y)$, or a loop or semi-edge at $\pi(x)$, which again contradicts that $S$ is independent.
    This shows that $\pi^{-1}(S)$ is an independent set in $G$.

    Let $x \in V(G) \setminus \pi^{-1}(S)$. We need to show that $x$ is incident to exactly $r$ ordinary edges leading to $\pi^{-1}(S)$. The covering map $\pi$ restricts to a bijection $\pi \res_{\Gamma(x)}$ on the dart neighbourhoods $\Gamma(x) \oto{\pi} \Gamma(\pi(x))$, and for any dart $d\in \Gamma(x)$, 
\begin{align*}
    &\phantom{\iff } d  \text{ leads to  } \pi^{-1}(S) \\
    &\iff \exists e \in D(G), \exists v \in \pi^{-1}(S) \text{ such that } d \text{ is paired to } e \text{ and } e \text{ is incident to } v \\
    &\implies \exists e \in D(G), \exists v \in \pi^{-1}(S) \text{ such that } \pi(d) \text{ is paired to } \pi(e) \text{ and } \pi(e) \text{ is incident to } \pi(v) \\
    &\hspace{5em} (\pi(e) \neq \pi(d) \text{ because they are incident to different vertices } \pi(x) \notin S,\ \pi(v)\in S ) \\ 
    &\implies \pi(d) \text{ leads to } S.
\end{align*}
    Conversely, if $d$ does not lead to $\pi^{-1}(S)$, then there are two cases:
    \begin{enumerate}
        \item If $d$ is a semi-edge or half of a loop, then the same holds for $\pi(d)$, so it does not lead to $\pi(S)$.
        \item If $d$ is paired to some $e$, then $e$ is incident to some vertex in $v \notin \pi^{-1}(S)$, so we know $\pi(d)$ is paired to $\pi(e)$ which is incident to $\pi(v) \notin S$
    \end{enumerate}
    Therefore, $d$ leads to $\pi^{-1}(S)$ if and only if $\pi(d)$ leads to $S$, so $\pi^{-1}(S)$ is an exact $r$-cover just as $S$ is.
\end{proof}

\begin{definition}[tensor/categorical product, see \cite{GTM207} Section 6.3]
Given two generalized graphs $G$ and $H$, their \emph{tensor product} $G \times H$ has dart set $D(G\times H) = D(G) \times D(H)$ and vertex set $V(G \times H) = V(G) \times V(H)$. The dart $(d, e) \in D(G \times H)$ is incident to $(u, v) \in V(G\times H)$ if and only if $d$ is incident to $u$ and $e$ is incident to $v$. Two darts $(d_1, d_2), \ (d_1', d_2') \in D(G) \times D(H)$ are paired if and only if $d_1$ is paired with $d_1'$ (or $d_1 = d_1'$ is a semi-edge) in $G$ and $d_2$ is paired with $d_2'$ (or $d_2=d_2'$ is a semi-edge) in $H$.

Note that if one of the graphs, say $G$, has no semi-edge or self-loop (a \emph{multigraph}), then $G \times H$ has no semi-edge or self-loop either. (If $d_1, d_1'$ are paired in $G$, then $d_1, d_1'$ are not identical, nor incident to the same vertex, so in $G\times H$, two darts $(d_1, d_2)$ and $ (d_1', d_2')$ that are paired cannot be the same dart (forming a semi-edge) or incident to the same vertex (forming a self-loop).) 
\end{definition}

In $G\times H$, the degree of the vertex $(u, v)$ is $d_G(u) d_H(v)$.

\begin{example}
    Theorem 6 of \cite{thequestion} gave independent exact $r$-covers $S_r$ in the grid graph $\ZZ^m$ for $r = 1, 2, m , 2m$. The examples are:
    \begin{align*}S_1 &= \left\{(x_1, \dots, x_m) \in \ZZ^m: \sum_{i=1}^m i x_i \equiv 0 \pmod{2m+1}
\right\}, \\ 
S_2 &= \left\{(x_1, \dots, x_m) \in \ZZ^m: \sum_{i=1}^m i x_i \equiv 0 \pmod{m+1}
\right\},\\
S_m &= \left\{(x_1, \dots, x_m) \in \ZZ^m: \sum_{i=1}^m \phantom{i} x_i \equiv 0 \pmod{3}
\right\},\quad \text{and}\\
S_{2m} &= \left\{(x_1, \dots, x_m) \in \ZZ^m: \sum_{i=1}^m  \phantom{i} x_i \equiv 0 \pmod{2}
\right\}.\end{align*}
    
    The first two sets $S_1$ and $S_2$ have previously appeared in Proposition 4.5.3 of \cite{metrebian} as constructions for \emph{$r$-covering $(2m/r+1)$-tilings} (which are independent exact $r$-covers with the additional property that $\ZZ^m$ can be partitioned into several translated copies of the cover).
    
    We can recover these known examples as lifts of independent exact $r$-covers of small graphs, as in \cref{lem:project}.
    \begin{enumerate}
        \item Consider the $2m$-regular tensor product graph $K_{2m+1} \times K_{2}$ whose vertex set we label as $\ZZ/((2m+1)\ZZ) \times \ZZ/(2\ZZ)$. Then $\{0\}\times \ZZ/(2\ZZ)$ is an independent exact $1$-cover. Its lift under the covering 
        \begin{align*}
        \pi\colon \ZZ^m &\to \ZZ/((2m+1)\ZZ) \times \ZZ/(2\ZZ) \\
        \vc x &\mapsto \left(\sum_{i=1}^m i x_i \mod{2m+1}, \quad \sum_{i=1}^m x_i \mod{2}\right)
        \end{align*}
        is the independent exact 1-cover $S_1$. 
        
        Similarly $\ZZ/((2m+1)\ZZ) \times \{0\}$ is an independent exact $2m$-cover, and its lift under $\pi$ is the independent exact $2m$-cover $S_{2m}$.
        
        \item Consider another $2m$-regular tensor product graph $K_{m+1} \times K_{3}$ whose vertex set we label as $\ZZ/((m+1)\ZZ) \times \ZZ/(3\ZZ)$. Then $\{0\}\times \ZZ/(3\ZZ)$ is an independent exact $2$-cover. Its lift under the covering 
        \begin{align*}
        \pi'\colon \ZZ^m &\to \ZZ/((m+1)\ZZ) \times \ZZ/(3\ZZ) \\
        \vc x &\mapsto \left(\sum_{i=1}^m i x_i \mod{m+1}, \quad \sum_{i=1}^m x_i \mod{3}\right)
        \end{align*}
        is the independent exact $2$-cover $S_2$.
        
        Similarly $\ZZ/((m+1)\ZZ) \times \{0\}$ is an independent exact $m$-cover, and its lift under $\pi'$ is the independent exact $m$-cover $S_m$.
    \end{enumerate}
    
\end{example}

In view of \cref{lem:project}, we would like to find a common covering of several graphs, each with their own independent exact covers. A classical result on finite common covering is as follows:

\begin{theorem}[\cite{commoncovering}, Theorem 1]
\label{thm:comcov}
If $G_1, G_2$ are finite $d$-regular simple graphs, then they share a finite common covering $G$ (that admits covering maps $\pi_1\colon G \to G_1$ and $\pi_2\colon G \to G_2$).
\end{theorem}

This has generalizations to non-regular graphs in \cite{gencommoncovering} but the regular case is sufficient for us.

When proving \cref{thm:comcov}, Angluin and Gardiner \cite{commoncovering} first gave a construction for 1-factorizable graphs:

\begin{lemma}[\cite{commoncovering}, Lemma 2]
\label{lem:1-fact}
If $G_1$ and $G_2$ are both $d$-regular simple graphs with a decomposition into $d$ 1-factors, then they have a common finite covering $G$ with $\abs{V(G_1)} \cdot \abs{V(G_2)}$ vertices. 
\end{lemma}
The common finite covering $G$ they constructed can also be decomposed into $d$ 1-factors, so the result generalizes to any $n > 2$ graphs that are 1-factorizable and $d$-regular.

Angluin and Gardiner \cite{commoncovering} also noted that if $G_i$ is not 1-factorizable, then one can take the canonical double covering $G_i \times K_2 \to G_i$. The graph $G_i \times K_2$ is bipartite and $d$-regular, whence 1-factorizable. Applying \cref{lem:1-fact} to $G_1 \times K_2$ and $G_2 \times K_2$, we obtain a common covering of $G_1, G_2$  of order $4\abs{V(G_1)}\abs{V(G_2)}$, but the construction is disconnected when applied to $G_i \times K_2$, and more specifically a disjoint union of 2 isomorphic subgraphs, so one subgraph is a common covering of order $2\abs{V(G_1)}\abs{V(G_2)}$. (See also Remark 6.2 in \cite{genleighton} for this upper bound.) 

Although the existence of finite common coverings is a known result (even for generalized graphs, Theorem 10 in \cite{semileighton}) and is sufficient for the proof of \cref{thm:main}, we provide in \cref{lem:1or2-fact} a construction with the same number of vertices as \cref{lem:1-fact}, for a larger class of pairs of generalized graphs $(G_1, G_2)$. This avoids the canonical double covering in some cases (and in particular, for generalized graphs arising in \cref{lem:simple,lem:compress}), allowing us to obtain optimal constructions when $d = 3, 4, 5, 6$, whose number of vertices matches the lower bounds in \cref{divisibility}. 

Without this new observation, one can still prove \cref{thm:main} from \cref{lem:1-fact}: the graph $G$ constructed in \cref{lem:simple} has no loops or repeated edges, so in case $G$ is not 1-factorizable or contains a semi-edge, we can first take the canonical double covering ($G\times K_2$ has no semi-edge and is 1-factorizable) before applying \cref{lem:1-fact}, but the overall construction obtained this way would not have minimal order.

\begin{lemma}
\label{lem:1or2-fact}
Let $G_1, G_2$ be finite $d$-regular generalized graphs. 
If for some $a, b \in \NN$ with $a+2b = d$, both $G_1$ and $G_2$ have a decomposition into $a$ 1-factors and $b$ families of disjoint cycles (in other words, $b$ 2-factors without semi-edges, but allowing loops and parallel edges), then $G_1$ and $G_2$ have a common finite covering $G$ with $\abs{V(G_1)} \cdot \abs{V(G_2)}$ vertices that can also be decomposed into $a$ 1-factors and $b$ families of disjoint cycles.
\end{lemma}

\begin{proof}
(When $b=0$, this construction is the same as \cite{commoncovering} Lemma 2. When $a=0$, the labelling of the edges is the same as \cite{schreier} Theorem 2, turning the graph into a Schreier coset graph. The whole proof is similar to that outlined in the Exercises in \cite{grosstucker} p.80.)

The graph $G$ we construct shall have vertex set $V(G)=V(G_1)\times V(G_2)$. We shall take a ``diagonal'' subgraph $G$ of the tensor product $G_1 \times G_2$, as follows:
\begin{enumerate}
    \item Fix $a+b$ distinct colours $c_i\ (i \in [a+b])$. In each graph, colour the 1-factors by colours $c_1, c_2, \ldots, c_a$ respectively, and colour the 2-factors by colours $c_{a+1}, \ldots, c_{a+b}$ respectively. 
    \item For each vertex $v \in G_i$ ($i = 1, 2$), and for each 1-factor with colour $c_j$ ($j\le a$), let $f_{ij}(v)$ be the neighbour of $v$ that is joined to $v$ by an edge of colour $c_j$, or $v$ itself if $v$ has a semi-edge of colour $c_j$.
    \item For each 2-factor, orient each cycle arbitrarily. 
    \item For each vertex $v \in G_i$ ($i = 1, 2$), and for each 2-factor with colour $c_j$ ($a < j\le a+b$),  let $f_{ij}(v)$ be the neighbour that the edge of colour $c_j$ from $v$ points to, according to the orientation we impose. 
    \item Define a new graph $G$ as follows: start with the vertex set $V(G_1) \times V(G_2)$, and we add coloured darts incident to each vertex $(u, v)$, pointing towards various vertices according to the following list (allow multiple darts towards the same neighbour, even with the same dart colour): 
    \begin{align*}
    &\text{Colour } c_1\colon \text{ towards } (f_{11}(u), f_{21}(v)), \\
    &\text{Colour } c_2\colon \text{ towards } (f_{12}(u), f_{22}(v)), \\ 
    &\vdots \\
    &\text{Colour } c_a\colon \text{ towards } (f_{1a}(u), f_{2a}(v)), \\
    &\text{Colour } c_{a+1}\colon \text{ towards } (f_{1,a+1}(u), f_{2,a+1}(v)) \text{ and } (f_{1,a+1}^{-1}(u), f_{2,a+1}^{-1}(v)), \\
    &\text{Colour } c_{a+2}\colon \text{ towards } (f_{1,a+2}(u), f_{2,a+2}(v)) \text{ and }  (f_{1,a+2}^{-1}(u), f_{2,a+2}^{-1}(v)), \\
    &\vdots \\
    &\text{Colour } c_{a+b}\colon \text{ towards } (f_{1,a+b}(u), f_{2,a+b}(v)) \text{ and }   (f_{1,a+b}^{-1}(u), f_{2,a+b}^{-1}(v)).
    \end{align*}
    \item Match the darts pointing at each other with the same colour. 
    
\end{enumerate}
    Note that the functions $f_{ij}$ are bijections for $i \in \{1, 2\}$ and all $j \le a+b$, and are self-inverse if $j \le a$, so all darts are paired, except in the case where the darts with colour $c_j$ ($j \le a$) incident to $u \in G_1$ and $v \in G_2$ are both semi-edges, and this is the only way we get a semi-edge of colour $c_j$ incident to $(u, v)\in G$. 
   
    We can see that the projection to each coordinate gives a valid graph covering $\pi_i: G \to G_i$, and that the colouring on $G$ gives a decomposition into $a$ 1-factors and $b$ families of disjoint cycles. 
\end{proof}

\begin{remark}
\label{remark:simple}
The construction in the proof of Lemma \ref{lem:1or2-fact} depends upon the decomposition, not just upon the graphs $G_1, G_2$, but no matter how we choose the factorizations, if one of $G_1$ or $G_2$ is a simple graph, then the resulting common covering will be simple. (The reason is that, if $\pi\colon G\to H$ is a covering map and $H$ is simple, then by \cref{prop:preimg}, $G$ is also simple.)
\end{remark}

\begin{corollary}
\label{cor:evencov}
Let $d$ be even. If $G_1, \ldots, G_n$ are $d$-regular generalized graphs without semi-edges (i.e.\ multigraphs with loops), then they have a common covering with $\prod_{i=1}^n \abs{V(G_i)} $ vertices. 
\end{corollary}
\begin{proof}
    By Petersen's 2-factor theorem (for multigraphs with loops), any regular graph of even degree has a $2$-factorization, so we can apply \cref{lem:1or2-fact} (with $a=0$ and $b=d/2$) iteratively.
\end{proof}

\begin{corollary}
\label{cor:oddcov}
Let $d$ be odd. If $G_1, \ldots, G_n$ are $d$-regular generalized graphs, and every $G_i$ has a 1-factor $M_i$ containing all semi-edges in $G_i$, then they have a common covering with $\prod_{i=1}^n \abs{V(G_i)} $ vertices.
\end{corollary}
\begin{proof}
    $G_i - M_i$ is $(d-1)$-regular without semi-edges, so it has a 2-factorization by Petersen's 2-factor theorem. This means we can apply \cref{lem:1or2-fact} (with $a=1$ and $b=(d-1)/2$) to  the first two graphs, and the common covering we obtain also has a 1-factor containing all semi-edges because of the observation in the proof of \cref{lem:1or2-fact} that semi-edges in the common covering only appear with the colour that has semi-edges in both graphs. We can then iterate this argument.
\end{proof}

\begin{remark}
    In \cite{nonunique}, Imrich and Pisanski showed that the least common covering (with fewest vertices) of two regular graphs may not be unique. It is also clear that the smallest possible number of vertices of a common covering is in general not the LCM, as there are non-isomorphic $d$-regular graphs with the same number of vertices.
\end{remark}

\section{Small graphs}
\label{smallgraphs}

In view of \cref{lem:project}, to prove \cref{thm:main} it suffices to exhibit a $d$-regular finite graph $G$ with an independent exact $r$-cover, one for each pair $(d, r)$ with $1 \le r \le d$, and then take a common covering of those graphs for common $d$ and various $r$, using \cref{thm:comcov} (or the quantitative \cref{cor:evencov,cor:oddcov}). 

To have just one independent exact $r$-cover is not a very strong condition: if we take a complete bipartite graph $K_{d, r}$, then the vertex class of size $r$ is an independent exact $r$-cover, and we can add an extra $(d-r)$-regular graph to the other vertex class of size $d$ to make the whole graph $d$-regular. The only case where this is not possible is when both $d-r$ and $d$ are odd. In that case we need to use one semi-edge. To make this precise:

\begin{lemma} 
\label{lem:simple}
For any $r \le d$, there exists a $d$-regular generalized graph $G$ on $d+r$ vertices with an independent exact $r$-cover but no loops or multiple edges, with the following additional property:
\begin{itemize}

    \item If $d$ is even, then we require that $G$ is simple (no semi-edge in addition to the above);
    \item If $d$ is odd, then we require that $G$ has a 1-factor $M$ such that every semi-edge is in $M$. 
\end{itemize}
\end{lemma}

\begin{proof}
We start with the complete bipartite graph $K_{d,r}$ between vertex classes  $A=\{0, 1,\dots, d - 1\}$ and 
$S=\{d, d+1, \ldots, d+r-1\}$. Now $S$ is an independent $r$-cover, every vertex in $S$ has degree $d$, and every vertex in $A$ has degree $r$, so it remains to put a $(d-r)$-regular generalized graph on $A$ (with $d$ vertices) satisfying the additional property in the Lemma.

Case 1: (\cref{fig:d-r even} illustrates an example.) If $d$ and $r$ are both odd or both even, then $d-r$ is even, so we can put a circulant graph on $A$ that is the Cayley graph on $A \cong \ZZ/(d\ZZ)$ with generating set $\{+1, -1, +2, -2, \ldots, +\frac{d-r}{2}, -\frac{d-r}{2}\}$. In the case where $d$ and $r$ are both odd, $d-r$ is even, so $\abs{A}-\abs{S}= d-r$ is even, and we can find the desired 1-factor $M$ consisting of edges
\begin{align*}
    &\{0, d\}, \{1, d+1\}, \ldots, \{r-1, d+r-1\}\ &&\text{between }A \text{ and } S, \text{ and}\\
    &\{r, r+1\}, \{r+2, r+3\}, \ldots, \{d-2, d-1\}  \ &&\text{within }A.
\end{align*}

Case 2: (\cref{fig:oddeven} illustrates an example.) If $d$ is even and $r$ is odd, then we can put the circulant graph on $A$ with generating set $\{+1, -1, +2, -2, \ldots, +\frac{d-r-1}{2}, -\frac{d-r-1}{2}, \pm\frac{d}{2}\}$, which is $(d-r)$-regular because the generator $ \pm\frac{d}{2}$ gives a matching within $A$.

Case 3: (\cref{fig:oddodd} illustrates an example.) If $d$ is odd and $r$ is even, then we first put the circulant graph on $A$ with generating set $\{+1, -1, +2, -2, \ldots, +\frac{d-r-1}{2}, -\frac{d-r-1}{2}\}$. Now each vertex has degree 1 less than the target. Since $r\ge 2$, the edges corresponding to the generator $+\frac{d-1}{2}$ in the circulant graph have not been added. We now add $\frac{d-1}{2}$ of them, namely 
    \[\left\{0, \frac{d-1}{2}\right\},\ \left\{1, \frac{d+1}{2}\right\},\ \ldots,\ \left\{\frac{d-3}{2}, d-2 \right\}\]
and put a semi-edge at $d-1$.

There is only one semi-edge $\{d-1\}$ and it forms a 1-factor together with the following edges:
    \begin{align*}
    &\{0, d\}, \{1, d+1\}, \ldots, \{r-1, d+r-1\}, \\
    &\{r, r+1\}, \{r+2, r+3\}, \ldots, \{d-3, d-2\}.  
    \end{align*}
\end{proof}

When $d$ and $r$ are not coprime, there are multigraphs with fewer that $d+r$ vertices with an independent exact $r$-cover.

\begin{lemma} 
\label{lem:compress}
Let $k = \gcd(d, r)$, then there exists a $d$-regular generalized graph $G$ on $(d+r)/k$ vertices with an independent exact $r$-cover, with the following additional property:
\begin{itemize}
    \item If $d$ is even, then we require that $G$ has no semi-edge.
    \item If $d$ is odd, then we require that $G$ has a 1-factor $M$ such that every semi-edge is in $M$. 

\end{itemize}
\end{lemma}
\begin{proof}
The construction is similar to \cref{lem:simple} with degree $d/k$ and an independent exact $r/k$-cover, but we replace each edge by $k$ parallel edges. The only problem is that we do not want to duplicate the semi-edges. 

We start with the complete bipartite graph $K_{d/k, r/k}$ between vertex classes $A=\{0, 1,\ldots,  d/k - 1\}$, 
$S=\{d/k, d/k+1, \ldots, (d+r)/k-1\}$, and replace each edge by $k$ parallel edges. Now $S$ is an independent $r$-cover, every vertex in $S$ has degree $d$, and every vertex in $A$ has degree $r$, so it remains to put a $(d-r)$-regular generalized graph on $A$ (with $d/k$ vertices) satisfying the additional property in the Lemma.

Case 1: If $d$ and $r$ are both odd or both even, then $d-r$ is even, so we can put a $d/k$-cycle (regular of degree 2) and replace each edge by $(d-r)/2$ parallel edges. In the case where $d$ and $r$ are both odd, $d-r$ is even but $k=\gcd(d, r)$ is odd, so $\abs{A}-\abs{S}= (d-r)/k$ is even, and we can find the desired 1-factor $M$ consisting of the edges:
\begin{align*}
    &\{0, d/k\}, \{1, d/k+1\}, \ldots, \{r/k-1, (d+r)/k-1\}\ &&\text{between }A \text{ and } S, \text{ and}\\
    &\{r/k, r/k+1\}, \{r/k+2, r/k+3\}, \ldots, \{d/k-2, d/k-1\}  \ &&\text{within }A.
\end{align*}

Case 2: If $d$ is even and $r$ is odd, then $k = \gcd(d, r)$ is odd and $d/k$ is even. So we can pair up the $d/k$ vertices in $A$ and put $d-r$ parallel edges between each pair.

Case 3: If $d$ is odd and $r$ is even, then $d-r$ and $d/k$ are odd. We put a $d/k$-cycle on $A$ and replace each edge by $(d-r-1)/2$ parallel edges. Now each vertex in $A$ has degree 1 less than the target. We add the edges $\{0,1\}, \{2,3\}, \ldots, \{d/k-3, d/k-2\}$ and a semi-edge at $d/k-1$.

There is only one semi-edge $\{d/k-1\}$ and it forms a 1-factor together with the following edges:
    \begin{align*}
    &\{0, d/k\}, \{1, d/k+1\}, \ldots, \{r/k-1, (d+r)/k-1\}, \\
    &\{r/k, r/k+1\}, \{r/k+2, r/k+3\}, \ldots, \{d/k-3, d/k-2\}.  
    \end{align*}

\end{proof}

\begin{figure}[p]
        \centering
        \begin{tikzpicture}[scale=\textwidth/10.2cm, roundnode/.style={circle, blue, draw=blue, very thin, inner sep = 2pt }]
         \pgfmathtruncatemacro{\d}{7};
         \pgfmathtruncatemacro{\r}{3};

         \pgfmathtruncatemacro{\last}{\d + \r - 1};

         % the indep set
         \foreach \i in {\d,...,\last}
{
\pgfmathtruncatemacro{\offset}{ (\d + \r)/2}
    \node[roundnode] (a\i) at  (\i - \offset ,0){$\i$};
}

         \pgfmathtruncatemacro{\dminus}{\d - 1};

% other vertices
\foreach \i in {0,...,\dminus}
{
	 \node[roundnode] (a\i) at (\i ,2) {$\i$};
}

    % the complete bipartite
    \foreach \i in {0,...,\dminus}{
        \foreach \j in {\d,...,\last}{
            \draw[ultra thin] (a\i) -- (a\j);
        }
    }

    % the circulant
    \foreach \step [parse=true, evaluate=\step] in {1,...,(\d-\r)/2}{ 
        \foreach \j [parse=true, evaluate=\j] in {0,...,\d - \step - 1}{
            \pgfmathtruncatemacro{\nxt}{\j + \step};
            \ifnum \step = 1
                \draw[ultra thin] (a\nxt) -- (a\j);
            \else 
                \draw[ultra thin] (a\j) to [bend right] (a\nxt);
            \fi
        }
        \foreach \j [parse=true, evaluate=\j] in {\d - \step, ..., \d - 1}{
            \pgfmathtruncatemacro{\intj}{\j}
            \pgfmathtruncatemacro{\nxt}{\j + \step - \d};
            \draw[ultra thin] (a\intj) to [out=160, in=20] (a\nxt);
        }
    }

    % the matching
    \foreach \j in {\d,...,\last}{
        \pgfmathtruncatemacro{\matched}{\j - \d};
        \draw[very thick,red] (a\matched) -- (a\j);
    }
    \foreach \j [parse=true, evaluate=\j] in {\r, \r +2,..., \d -2}{
        \pgfmathtruncatemacro{\jplus}{\j + 1};
        \draw[very thick,red] (a\j) -- (a\jplus);
    }

    \node[label=east:$(d-1)$] at (a\dminus){};
    \node[label=south:$(d)$] at (a\d){};
    \node[label=south:$(d+r-1)$] at (a\last){};
    
	\end{tikzpicture}
        \caption{(Case 1 in \cref{lem:simple}) A regular simple graph of degree $d=7$, where $\{7,8,9\}$ is an independent exact $3$-cover. A perfect matching is indicated in red. }
        \label{fig:d-r even}
    \end{figure}
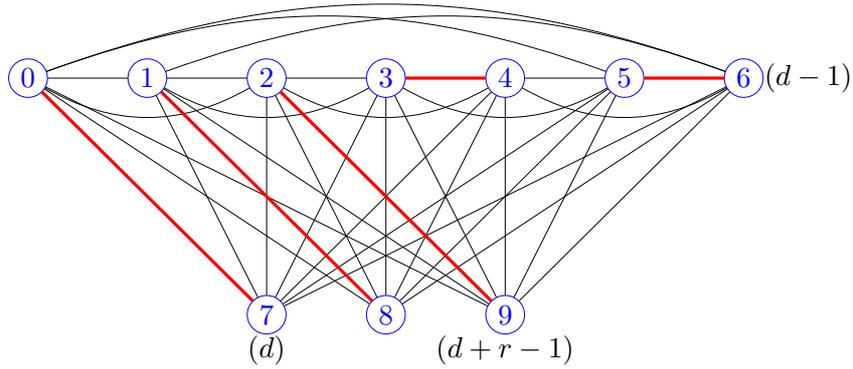
    
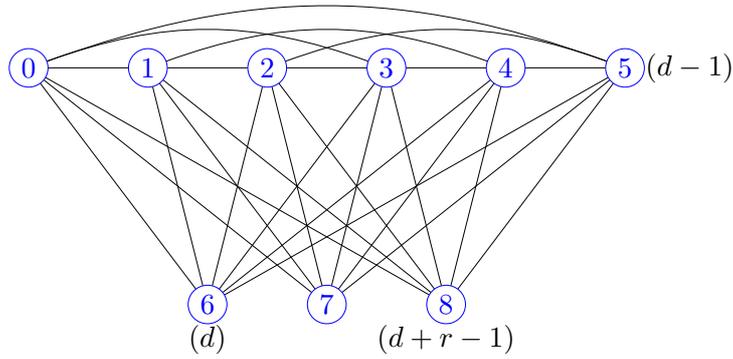
\begin{figure}[p]
        \centering
        \begin{tikzpicture}[scale=\textwidth/10.2cm, roundnode/.style={circle, blue, draw=blue, very thin, inner sep = 2pt }]
         \pgfmathtruncatemacro{\d}{6};
         \pgfmathtruncatemacro{\r}{3};

         \pgfmathtruncatemacro{\last}{\d + \r - 1};

         % the indep set
         \foreach \i in {\d,...,\last}
{
\pgfmathsetmacro{\offset}{ (\d + \r)/2}
    \node[roundnode] (a\i) at  (\i - \offset ,0){$\i$};
}

         \pgfmathtruncatemacro{\dminus}{\d - 1};

% other vertices
\foreach \i in {0,...,\dminus}
{
	 \node[roundnode] (a\i) at (\i ,2) {$\i$};
}

    % the complete bipartite
    \foreach \i in {0,...,\dminus}{
        \foreach \j in {\d,...,\last}{
            \draw[ultra thin] (a\i) -- (a\j);
        }
    }

    % the circulant
    \foreach \step [parse=true, evaluate=\step] in {1,...,(\d-\r)/2}{ 
        \foreach \j [parse=true, evaluate=\j] in {0,...,\d - \step - 1}{
            \pgfmathtruncatemacro{\nxt}{\j + \step};
            \ifnum \step = 1
                \draw[ultra thin] (a\nxt) -- (a\j);
            \else 
                \draw[ultra thin] (a\j) to [bend right] (a\nxt);
            \fi
        }
        \foreach \j [parse=true, evaluate=\j] in {\d - \step, ..., \d - 1}{
            \pgfmathtruncatemacro{\intj}{\j}
            \pgfmathtruncatemacro{\nxt}{\j + \step - \d};
            \draw[ultra thin] (a\intj) to [out=160, in=20] (a\nxt);
        }
    }

     \foreach \j [parse=true, evaluate=\j] in {0,...,\d - \d/2 - 1}{
        \pgfmathtruncatemacro{\nxt}{\j + \d/2};
        \draw[ultra thin] (a\j) to [out=20, in=160] (a\nxt);
    }

    \node[label=east:$(d-1)$] at (a\dminus){};
    \node[label=south:$(d)$] at (a\d){};
    \node[label=south:$(d+r-1)$] at (a\last){};
    
	\end{tikzpicture}
        \caption{(Case 2 in \cref{lem:simple}) A regular simple graph of degree $d=6$, where $\{6,7,8\}$ is an independent exact $3$-cover. }
        \label{fig:oddeven}
    \end{figure}
    
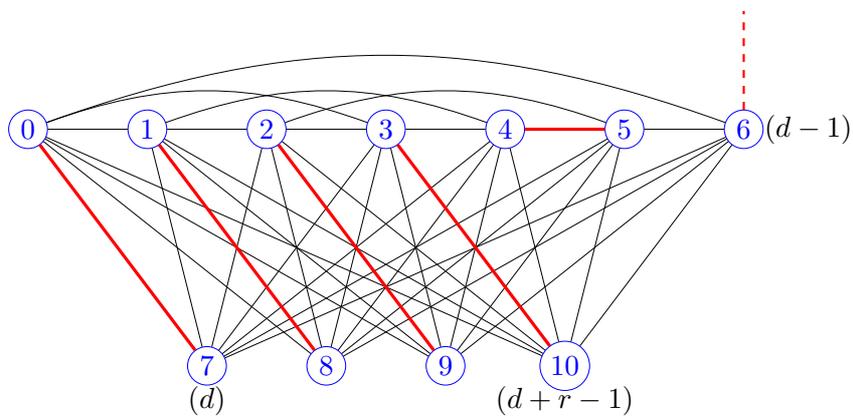
\begin{figure}[p]
        \centering
        \begin{tikzpicture}[scale=\textwidth/10.2cm, roundnode/.style={circle, blue, draw=blue, very thin, inner sep = 2pt }]
         \pgfmathtruncatemacro{\d}{7};
         \pgfmathtruncatemacro{\r}{4};

         \pgfmathtruncatemacro{\last}{\d + \r - 1};

         % the indep set
         \foreach \i in {\d,...,\last}
{
\pgfmathsetmacro{\offset}{ (\d +\r)*0.5}
    \node[roundnode] (a\i) at  (\i - \offset ,0){$\i$};
}

         \pgfmathtruncatemacro{\dminus}{\d - 1};

% other vertices
\foreach \i in {0,...,\dminus}
{
	 \node[roundnode] (a\i) at (\i ,2) {$\i$};
}

    % the complete bipartite
    \foreach \i in {0,...,\dminus}{
        \foreach \j in {\d,...,\last}{
            \draw[ultra thin] (a\i) -- (a\j);
        }
    }

    % the circulant
    \pgfmathtruncatemacro{\halfd}{(\d-1)/2}
    
    \foreach \step [parse=true, evaluate=\step] in {1,...,(\d-\r)/2}{ 
        \foreach \j [parse=true, evaluate=\j] in {0,...,\d - \step - 1}{
            \pgfmathtruncatemacro{\nxt}{\j + \step};
            \ifnum \step = 1
                \draw[ultra thin] (a\nxt) -- (a\j);
            \else 
                \draw[ultra thin] (a\j) to [bend right] (a\nxt);
            \fi
        }
            \foreach \j [parse=true, evaluate=\j] in {\d - \step, ..., \d - 1}{
                \pgfmathtruncatemacro{\intj}{\j}
                \pgfmathtruncatemacro{\nxt}{\j + \step - \d};
                \draw[ultra thin] (a\intj) to [out=160, in=20] (a\nxt);
            }
    
    }

     \foreach \j [parse=true, evaluate=\j] in {0,...,\d - \halfd - 2}{
            \pgfmathtruncatemacro{\nxt}{\j + \halfd};
             \draw[ultra thin] (a\j) to [out=20, in=160] (a\nxt);
    }
    
    % the semi-edges
    \draw[red, thick, dashed] (a\dminus) -- (\dminus, 3);
    % the matching
    \foreach \j in {\d,...,\last}{
        \pgfmathtruncatemacro{\matched}{\j - \d};
        \draw[very thick,red] (a\matched) -- (a\j);
    }
    \foreach \j [parse=true, evaluate=\j] in {\r, \r +2,..., \d -3}{
        \pgfmathtruncatemacro{\jplus}{\j + 1};
        \pgfmathtruncatemacro{\jtrunc}{\j};
        \pgfmathtruncatemacro{\dminustwo}{\d - 2};
        \ifnum \jtrunc < \dminustwo
            \draw[very thick,red] (a\jtrunc) -- (a\jplus);
        \fi
    }

    \node[label=east:$(d-1)$] at (a\dminus){};
    \node[label=south:$(d)$] at (a\d){};
    \node[label=south:$(d+r-1)$] at (a\last){};
    
	\end{tikzpicture}
        \caption{(Case 3 in \cref{lem:simple}) A regular graph of degree $d=7$ with semi-edges (dashed lines), where $\{7,8,9, 10\}$ is an independent exact $4$-cover. A 1-factor is indicated in red. }
        \label{fig:oddodd}
    \end{figure}

\vspace{-2em}
\section{Divisibility conditions and optimal constructions for $d \le 6$ }
\label{divisibility}

Following \cref{thm:main}, Johnson asked:
\begin{question}[Johnson 2024, personal communication]
\label{q:div}
    Let $d\in \NN$. What is the smallest order $n$ of a $d$-regular graph $G$ such that, for all $r \in [d]$, $G$ has an independent exact $r$-cover $S_r$? In particular, from \cref{def:indepexact}, there is a divisibility condition requiring that
    \[\frac{rn}{d+r} \in \NN \quad \forall r \in [d]. \tag{$\star$}\]
    Does there exist a graph with the property above, whose order is the smallest $n$ satisfying ($\star$)?
\end{question}

This lower bound is tight for $d=4$:
\begin{proposition}
The smallest order of a 4-regular simple graph with independent exact $r$-covers for all $r \le 4$ is 210.
\end{proposition}
\begin{proof}
    The divisibility condition for independent exact 1-, 2-, 3-, and 4-covers says $5 \mid n$, $3\mid n$, $7 \mid n$, and $2 \mid n$, respectively, so $210 \mid n$. For the construction, we apply \cref{cor:evencov} to obtain a common covering of the following four 4-regular graphs:

\begin{enumerate}[nosep, label={for $r = \arabic*$,}, left=2em]
    \item $K_5$; 
    \item $K_3$ with each edge replaced by two parallel edges; 
    \item a 7-vertex graph from \cref{lem:simple} Case 2, i.e., a $K_{4,3}$ with a matching of size two added within the independent set of size 4;
    \item the dipole $D_4$ (four parallel edges between two vertices).
\end{enumerate}
The result is a simple graph because $K_5$ is simple (and because of Remark \ref{remark:simple}), it has an independent exact $r$-cover for all $r \leq 4$ by Lemma \ref{lem:project}, and it has
$5\times 3 \times 7 \times 2 = 210$ vertices. 
\end{proof}

In the $d=3$ case, the divisibility bound says $20 \mid n$, and \cite{thequestion} gives a construction on 40 vertices. We give a slightly stronger divisibility condition, showing that 40 is the smallest possible. In particular, this gives a negative answer to the last part of \cref{q:div}, and shows that $(\star)$ is not the only constraint on $n$. 

\begin{lemma}
\label{lem:2indep}
    Suppose $G$ has $n$ vertices, $S_1$ is an independent exact $r_1$-cover of $G$,  $S_2$ is an independent exact $r_2$-cover of $G$, and $r_1 \neq r_2$. Then 
    \[\abs{S_1\cap S_2} = \frac{\abs{S_1}\abs{S_2}}{n} = \frac{r_1 r_2 n}{(d+r_1) (d+r_2)}. \]
\end{lemma}
\begin{proof}
    Let $\abs{S_1\cap S_2} = x$. Since $S_2$ is an exact $r_2$-cover, every vertex in $S_1\setminus S_2$ has exactly $r_2$ neighbours in $S_2$, but $S_1$ is an independent set, so these $r_2$ neighbours are in $S_2 \setminus S_1$.

    Similarly every vertex in $S_2 \setminus S_1$ has exactly $r_1$ neighbours in $S_1 \setminus S_2$. Therefore the number of edges between  $S_1 \setminus S_2$ and $S_2\setminus S_1$ can be counted in two ways to give
    \[(\abs{S_1} - x) r_2 = (\abs{S_2} - x) r_1. \]
    We know $\abs{S_1} = r_1n/(d+r_1)$ and $\abs{S_2} = r_2 n / (d+r_2)$, and solving for $x$ gives the result.
\end{proof}

Putting $d=3$, $r_1 = 1$ and $r_2 = 3$, we see that $\abs{S_1\cap S_2} = n/8$, so $n=20$ is not possible and the $n=40$ example in \cite{thequestion} is minimal.

\cref{lem:2indep} also gives the sharp lower bound for $d=6$.

\begin{proposition}
The smallest order of a 6-regular simple graph with independent exact $r$-covers for all $r \le 6$ is 9240.
\end{proposition}

\begin{proof} Using \cref{lem:2indep} with $r_1 = 2$ and $r_2 = 6$, we have $2^3\mid n$. Also from the existence of independent exact 1-cover, 3-cover, 4-cover, and 5-cover respectively, we have $7 \mid n$, $3 \mid n$, $5 \mid n$, and $11\mid n$. So $9240 = 2^3 \times 3 \times 5 \times 7 \times 11 \mid n$.

For the construction, we apply \cref{cor:evencov} to obtain a common covering of the following six 6-regular graphs:
\begin{enumerate}[nosep, label={for $r = \arabic*$,}, left=2em]
    \item $K_7$;
    \item $K_4$ with each edge doubled into a pair of parallel edges;
    \item $K_3$ with each edge replaced by three parallel edges;
    \item $K_{2, 3}$ with each edge doubled, and adding a 3-cycle on the vertex class of size 3 (from \cref{lem:compress} Case 1);
    \item an 11-vertex graph from \cref{lem:simple} Case 2 (i.e., a $K_{6,5}$ with a matching added into the independent set of size 6);
    \item the dipole $D_6$ (six parallel edges between two vertices).
\end{enumerate}
The result is a simple graph because $K_7$ is simple (see Remark \ref{remark:simple}), it has an independent exact $r$-cover for all $r \leq 6$ by Lemma \ref{lem:project}, and it has $7\times 4 \times 3 \times 5 \times 11 \times 2 = 9240$ vertices.
\end{proof}

One might expect similar results to hold for three or more independent exact covers of distinct sizes, that when a vertex is chosen uniformly at random, the independent exact covers correspond to independent (not just pairwise) events , and 
\[\abs{S_1 \cap S_2 \cap \cdots \cap S_k} = \frac{\abs{S_1}\abs{S_2}\cdots \abs{S_k}}{n^{k-1}}\]
if each $S_i$ is an independent exact $r_i$-cover in $G$ with the $r_i$'s pairwise distinct.
This would be the case if $G$ is obtained from a product construction like \cref{lem:1or2-fact} and the $S_i$'s are lifted from independent exact $r_i$-covers of smaller graphs.
However, this conjecture does not hold in general. In fact, we can  force $S_1 \cap S_2 \cap S_3 = S_2 \cap S_3$, as in \cref{ex:3nonindep}.

\begin{table}[h!]
\centering
\begin{tabular}{rr|r|r|r|r|r|r|r|r|}
          &              & $V_{000}$ & $V_{001}$ & $V_{010}$ & $V_{011}$ & $V_{100}$ & $V_{101}$ & $V_{110}$ & $V_{111}$ \\ 
          & $\#$vertices & 735       & 630       & 210       & 63        & 315       & 140       & 0         & 91        \\ \hline
$V_{000}$ & 735          & 4/147     & 1/15      & 1/105     & 2/21      & 2/105     & 4/35      &           & 1/7       \\ \hline
$V_{001}$ & 630          & 1/15      &           & 1/10      &           & 1/9       &           &           &           \\ \hline
$V_{010}$ & 210          & 1/105     & 1/10      &           &           & 1/15      & 1/10      &           &           \\ \hline
$V_{011}$ & 63           & 2/21      &           &           &           & 1/9       &           &           &           \\ \hline
$V_{100}$ & 315          & 2/105     & 1/9       & 1/15      & 1/9       &           &           &           &           \\ \hline
$V_{101}$ & 140          & 4/35      &           & 1/10      &           &           &           &           &           \\ \hline
$V_{110}$ & 0            &           &           &           &           &           &           &           &           \\ \hline
$V_{111}$ & 91           & 1/7       &           &           &           &           &           &           &           \\ \hline
\end{tabular}
\caption{Densities of biregular bipartite subgraphs induced between vertex classes in \cref{ex:3nonindep}. } \label{table:threeindepcover}
\end{table}

\begin{example}
\label{ex:3nonindep}
We shall construct a regular graph of degree $d = 105$ and order $n=2184$, with an independent exact 
 $r_i$-cover $S_i$ for each $i = 1, 2, 3$, where $r_1 = 77$, $r_2 = 21$, $r_3 = 35$, such that $S_2 \cap S_3 \subseteq S_1$, and hence 
\[\abs{S_1 \cap S_2 \cap S_3} = \abs{S_2 \cap S_3} = \frac{\abs{S_2}\abs{S_3}}{n} > \frac{\abs{S_1}\abs{S_2}\abs{S_3}}{n^2}. 
\]

We partition the vertex set into 8 parts $V_{000}, V_{001}, V_{010}, \ldots, V_{111}$ of sizes as shown in Table \ref{table:threeindepcover}. We put any $20$-regular graph (e.g.\ a circulant graph) on $V_{000}$ (represented as ``bipartite density'' $20/735 = 4/147$ in the table). All other parts will be an independent set respectively. Across any two parts, we put a biregular graph between them, with bipartite density as shown in Table \ref{table:threeindepcover}. For example, the biregular bipartite subgraph induced between $V_{010}$ and $V_{001}$ has density $1/10$, meaning each vertex from $V_{010}$ has 63 neighbours in $V_{001}$, and each vertex from $V_{001}$ has 21 neighbours in $V_{010}$. (Blank cells represent density zero.) Such biregular graphs exist because the densities are chosen such that the degrees on both sides are integers. 

Summing the degrees for each row and column, one can verify that the graph is 105-regular.

The independent exact covers are 
\begin{align*}
    S_1 &= V_{001} \cup V_{011} \cup V_{101} \cup V_{111}, \\
    S_2 &= V_{010} \cup V_{011} \cup V_{110} \cup V_{111}, \\
    S_3 &= V_{100} \cup V_{101} \cup V_{110} \cup V_{111}.
\end{align*}
In each row, we can verify that the degrees in the four columns corresponding to each $S_i$ sum to $r_1 = 77$, $r_2 = 21$, and $r_3 = 35$, respectively. This shows that $S_i$ is an exact $r_i$-cover.
\end{example}

Despite this counterexample for general $r_1, r_2, r_3$, the conjecture holds when one of them is equal to the degree $d$, because the independent exact $d$-cover and its complement must form a balanced bipartition of the graph.

\begin{lemma} 
\label{lem:3indep} 
    Let $G$ be a $d$-regular graph on $n$ vertices, $S_1$ be an independent exact $r_1$-cover of $G$,  $S_2$ be an independent exact $r_2$-cover of $G$, $S_3$ be an independent exact $d$-cover of $G$. Suppose $r_1 \neq r_2$ are both distinct from $d$. Then 
    \[\abs{S_1\cap S_2 \cap S_3} = \frac{\abs{S_1}\abs{S_2}\abs{S_3}}{n^2} = \frac{r_1 r_2 n}{2 (d+r_1) (d+r_2)}. \]
\end{lemma}
\begin{proof}
    Let $\abs{S_1\cap S_2 \cap S_3} = x$.

    Since $S_2$ is an exact $r_2$-cover, every vertex in $S_1 \cap S_3 \setminus S_2$ has exactly $r_2$ neighbours in $S_2$, but $S_1$ and $S_3$  are independent sets, so these $r_2$ neighbours are in $S_2  \setminus (S_1\cup S_3)$. 

    Similarly every vertex in $S_2 \setminus (S_1\cup S_3)$ has exactly $r_1$ neighbours in $S_1$. By our assumption, both $S_2$ and $V(G) \setminus S_3$ are independent sets. (The latter is independent because if $G$ has an independent exact $d$-cover $S_3$, then $G$ must be bipartite with parts $S_3$ and $V(G)\setminus S_3$.) Therefore these $r_1$ neighbours must be in $S_1 \cap S_3 \setminus S_2$, and the number of edges between  $S_1 \cap S_3 \setminus S_2$ and $S_2 \setminus ( S_1 \cup S_3)$ can be counted in two ways to give
        \[\abs{S_1 \cap S_3 \setminus S_2} r_2 = \abs{S_2 \setminus ( S_1 \cup S_3)} r_1, \]
    so
    \[(\abs{S_1\cap S_3}  - \abs{S_1 \cap S_2 \cap S_3}) r_2 = (\abs{S_2\setminus S_3} - \abs{S_1 \cap S_2} +\abs{S_1\cap S_2 \cap S_3} ) r_1. \]
    Applying \cref{lem:2indep} to the pairwise intersections, we have 
    \[\left( \frac{r_1 n}{2(d+r_1)} - x\right) r_2 = \left( \frac{r_2 n}{2(d+r_2)} - \frac{r_1 r_2 n}{(d+r_1)(d+r_2)} + x \right) r_1, \]
    and the result follows from solving for $x$.
\end{proof}

\cref{lem:3indep} gives the sharp lower bound for $d=5$.

\begin{proposition}
The smallest order of a 5-regular simple graph with independent exact $r$-covers for all $r \le 5$ is 6048.
\end{proposition}

\begin{proof} Taking $r_1 = 1$ and $r_2 = 3$ in \cref{lem:3indep}, we have $2^5 \mid n$. Using the earlier \cref{lem:2indep} with $r_1 = 1$ and $r_2 = 4$, we have $3^3\mid n$. Also $7 \mid n$ from the existence of an independent exact $2$-cover. Hence $6048 = 2^5 \times 3^3 \times 7 \mid n$.

For the construction, we apply \cref{cor:oddcov} to obtain a common covering of the following five 5-regular graphs:

\begin{enumerate}[nosep, label={for $r = \arabic*$,}, left=2em]
    \item $K_6$; 
    \item a 7-vertex graph with semi-edge from \cref{lem:simple} Case 3;
    \item an 8-vertex graph from \cref{lem:simple} Case 1, i.e., a $K_{5,3}$ with a 5-cycle added within the independent set of size five;
    \item a 9-vertex graph with semi-edge from \cref{lem:simple} Case 3, i.e., a $K_{5,4}$ with a 1-factor added to the independent set of size 5 (the 1-factor consisting of a matching of size two together with one semiedge);
    \item the dipole $D_5$ (five parallel edges between two vertices).
\end{enumerate}

The result is a simple graph because $K_6$ is simple, it has an independent exact $r$-cover for all $r \leq 5$ by Lemma \ref{lem:project}, and it has
$6\times 7 \times 8 \times 9 \times 2 = 6048$ vertices. 
\end{proof}

In the search of smallest graphs with independent exact $r$-covers for all $r \le d$, $d=7$ is therefore the smallest open case, for which we have a gap of a factor of 2 between the lower and upper bounds. For larger $d$, the gap between our lower and upper bounds grows superexponentially:

\begin{proposition}
Let $N(d)$ be the smallest order of a $d$-regular simple graph with an independent exact $r$-cover for all $r \le d$. Then
\[d(2-o(1)) \le \log N(d) \le d (\log (4d) -1 +o(1)),\]
where $\log$ is the natural logarithm.
\end{proposition}
\begin{proof}
    Let $G$ be an optimal $d$-regular graph satisfying the condition, with $\abs{G} = N(d)$. 
    First we compute a lower bound from the divisibility conditions. Using the fact that any independent exact $r$-cover has size $\frac{r}{d+r} N(d)$, in \cref{lem:3indep} we also have 
    \[\abs{S_3\setminus(S_1\cup S_2)} = \frac{d^2 N(d)}{2 (d+r_1)(d+r_2)},\]
    so 
    \[2\lcm \{(d+r_1)(d+r_2): 0 \le r_1 <  r_2 \le d \} \ \text{ divides }\ d^2 N(d),\]
    and 
    \[N(d) \ge \frac{2}{d^2} \lcm \{(d+r_1)(d+r_2): 0 \le r_1 <  r_2 \le d \} \tag{$\Diamond$}\]

    For a more explicit lower bound we note that $\operatorname{lcm} \{(d+r_1)(d+r_2): 0 \le r_1 <  r_2 \le d \}$ is at least $\lcm(d+1, d+2, \ldots, 2d) = \lcm(1, 2, \ldots, 2d) = \exp{\psi(2d)}$ where $\psi$ is the second Chebyshev function satisfying $\psi(2d)/(2d) = 1+o(1).$ (See e.g.\ Ch. 4 of \cite{apostol})
    
    The upper bound from \cref{cor:evencov,cor:oddcov} is the product of $d+1, d+2, \ldots, 2d$ which is $(2d)!/d! = \exp(d(\log (4d) -1+o(1)))$ by Stirling's approximation. 
\end{proof}

\begin{remark}
The lcm in the lower bound $(\Diamond)$ is larger than $\exp \psi(2d)$, but it satisfies
\[ \operatorname{lcm} \{(d+r_1)(d+r_2): 0 \le r_1 <  r_2 \le d \} \leq (\operatorname{lcm}\{1,2,\ldots,2d\})^2 = \exp(2\psi(2d)) = \exp((4+o(1))d),
\]
so even if the lower bound is done more carefully with $(\Diamond)$, we can at most improve the constant in the exponent, but to close the super-exponential gap will require a new argument.
\end{remark}

It would be interesting to determine the growth rate of $N(d)$. We do not know whether there are more efficient ways to construct common coverings, and whether there are constructions with a lot of independent exact $r$-covers not coming from common coverings.

\section*{Acknowledgement}
The author would like to thank Robert Johnson for his comments on an earlier draft and for suggesting the question that leads to \cref{divisibility}, and David Ellis for his advice on the exposition.

\end{document}